\pdfoutput=1

\documentclass{amsart}

\usepackage[T1]{fontenc}

\usepackage[
    backend=biber,
    style=numeric,
    sorting=nyt,
    doi=false,
    isbn=false
]{biblatex}
\usepackage{enumitem}
\usepackage{hyperref}
\usepackage[capitalise]{cleveref}
\newlist{partenum}{enumerate}{1}
\setlist[partenum,1]{label=(\arabic*)}
\crefname{partenumi}{part}{parts}

\theoremstyle{plain}
\newtheorem*{theorem*}{Theorem}
\newtheorem{theorem}{Theorem}[section]
\newtheorem{proposition}[theorem]{Proposition}
\newtheorem{lemma}[theorem]{Lemma}
\newtheorem{corollary}[theorem]{Corollary}
\theoremstyle{definition}
\newtheorem{definition}[theorem]{Definition}
\newtheorem{example}[theorem]{Example}
\theoremstyle{remark}
\newtheorem{remark}[theorem]{Remark}

\usepackage{amsmath}
\usepackage{amssymb}
\usepackage{amsthm}
\usepackage[all,cmtip]{xy}

\renewcommand{\b}{\mathrm{b}}
\newcommand{\h}{\mathrm{h}}
\renewcommand{\t}{\mathrm{t}}
\newcommand{\w}{\mathrm{w}}
\newcommand{\op}{\mathrm{op}}
\renewcommand{\L}{\mathrm{L}}

\newcommand{\add}{\textnormal{add}}
\newcommand{\ex}{\textnormal{ex}}
\newcommand{\fin}{\textnormal{fin}}
\newcommand{\wex}{\textnormal{w-ex}}

\newcommand{\Fun}{\operatorname{Fun}}
\newcommand{\Hom}{\operatorname{Hom}}

\newcommand{\Kb}{\operatorname{K}^{\b}}
\newcommand{\Ndg}{\operatorname{N}_{\textnormal{dg}}}
\renewcommand{\P}{\operatorname{\cat{P}}}
\newcommand{\Sp}{\operatorname{Sp}}

\newcommand{\cat}[1]{\mathcal{#1}}
\newcommand{\Cat}[1]{\mathsf{#1}}
\newcommand{\cls}[1]{\mathbf{#1}}

\title{The weight complex functor is symmetric monoidal}
\author{Ko Aoki}
\address{
    Department of Mathematics,
    Tokyo Institute of Technology, 2--12--1 \=Ookayama, Meguro-ku,
    Tokyo 152--8551, Japan
}
\email{aoki.k.an@m.titech.ac.jp}
\date{\today}

\begin{document}

\begin{abstract}
    Bondarko's (strong) weight complex functor
    is a triangulated functor from Voevodsky's triangulated
    category of motives
    to the homotopy category of chain complexes of classical Chow motives.
    Its construction is valid for any dg enhanced triangulated
    category equipped with a weight structure. In this paper we consider
    weight complex functors in the setting of stable
    symmetric monoidal $ \infty $-categories.
    We prove that the weight complex functor is symmetric monoidal under
    a natural compatibility assumption. To prove this result,
    we develop additive and stable symmetric monoidal
    variants of the $ \infty $-categorical
    Yoneda embedding, which may be of independent interest.
\end{abstract}

\maketitle

\section{Introduction}\label{9d44b04442}

In the paper \autocite{Bondarko10}, Bondarko introduced
the notion of a weight structure
(see \cref{8d18b2095c}) on a triangulated category
as a variant of $ \t $-structure
and then constructed a (strong) weight complex functor when the triangulated category
has a dg enhancement.
One primary example of a weight structure is
the motivic weight structure on
$ \operatorname{DM}_{\textnormal{gm}}^{(\textnormal{eff})}(k;\mathbb{Q}) $,
whose existence was proven in \autocite[Section~6]{Bondarko10}.
The weight complex functor associated to that is
a functor $ \operatorname{DM}_{\textnormal{gm}}^{(\textnormal{eff})}(k;\mathbb{Q})
\to \Kb(\operatorname{Chow}^{(\textnormal{eff})}(k;\mathbb{Q})) $, which was
studied in \autocite[Section~6]{Bondarko09}.
In a recent preprint \autocite{Sosnilo17}, Sosnilo considered weight structures using
stable $ \infty $-categories and showed that
the weight complex can be constructed in that setting.
In this paper we consider them with symmetric monoidal structures.

Our result mentioned in the title is \cref{7704dfff4b},
which states that the weight complex functor
is symmetric monoidal under a natural compatibility condition.
We note that its dg variant appeared
in \autocite[Lemma~20]{Bachmann},
but the author was unable to fill in the details of the proof
presented there.

Applying this to the motivic weight structure,
we have the following:

\begin{theorem*}
    Let $ k $ be a perfect field.
    Then the weight complex functor
    $ \operatorname{DM}_{\textnormal{gm}}^{(\textnormal{eff})}(k;\mathbb{Q})
    \to \Kb(\operatorname{Chow}^{(\textnormal{eff})}(k;\mathbb{Q})) $ is
    symmetric monoidal.
\end{theorem*}

We note that this theorem is mentioned in \autocite[Remark~37]{Kelly} as
a desired statement which seems to have no written proof.

\subsection*{Outline}

We begin in \cref{5eff74713e}
by recalling some $ \infty $-category theory
which we will need in this paper.
There we give
the definition of an additive symmetric monoidal $ \infty $-category
and prove a version of the Yoneda embedding for it.
In \cref{cca8147ee1}, we review the theory of weight structures
on stable $ \infty $-categories and define the weight complex functor.
\Cref{17df43aa2c} is the main part of this paper.
There we introduce a notion of compatibility between
a symmetric monoidal structure and a weight structure
and prove the main result.

\subsection*{Acknowledgments}

I would like to thank
Thomas Nikolaus for patiently answering my question on
his multiplicative Yoneda lemma.
I am also grateful to Tom Bachmann for his useful comment on
an early draft which reduced the length of this paper.
I also thank
Shane Kelly for suggesting me this problem,
discussing it with me in the dg context,
and making helpful comments which improved the readability of this paper.

\section{Preliminaries from $ \infty $-category theory}\label{5eff74713e}

When we deal with $ \infty $-categories,
we generally follow terminologies and notations used in \autocite{HTT, HA},
but we will regard every category as an $ \infty $-category by taking its nerve.

\subsection{Additive and stable $ \infty $-categories}

We refer readers to \autocite[Chapter~1]{HA} for the theory of stable $ \infty $-categories
and \autocite{GepnerGrothNikolaus15} for the theory of additive $ \infty $-categories.

\begin{definition}\label{f5eeb1acf1}
    Let $ \cat{A} $ be an $ \infty $-category.
    We call $ \cat{A} $ \emph{additive} if it has
    finite products and coproducts and
    its homotopy category is an additive category.

    Let $ \cat{A} $ and~$ \cat{A}' $ be additive $ \infty $-categories and
    $ f \colon \cat{A} \to \cat{A}' $ a functor between them.
    We call $ f $ \emph{additive} if it preserves finite products (or coproducts,
    equivalently).
    We write $ \Fun^{\add}(\cat{A}, \cat{A}') $ for
    the full subcategory of $ \Fun(\cat{A}, \cat{A}') $ spanned by
    additive functors.
\end{definition}

\begin{example}\label{a466051a25}
    Every stable $ \infty $-category is additive.
    More generally, a full subcategory of a stable $ \infty $-category
    closed under finite (co)products is additive.

    Another example of an additive $ \infty $-category is
    (the nerve of) an additive category.
    For an additive $ \infty $-category~$ \cat{A} $,
    the canonical functor
    $ \cat{A} \to \h\cat{A} $ is additive.
\end{example}

Let $ \Cat{Cat}_{\infty}^{\add} $ denote the subcategory of
the large $ \infty $-category of small $ \infty $-categories $ \Cat{Cat}_{\infty} $
whose objects are additive $ \infty $-categories
and morphisms are additive functors.
Similarly, we let $ \Cat{Cat}_{\infty}^{\ex} $ denote the subcategory
of $ \Cat{Cat}_{\infty} $ whose objects are stable $ \infty $-categories and
morphisms are exact functors.

We recall a stable version of the Yoneda embedding.
Here $ \Cat{S} $ and~$ \Cat{Sp} $ denote the large $ \infty $-categories of
spaces and spectra, respectively.

\begin{definition}\label{eab25264e8}
    Let $ \cat{C} $ be a stable $ \infty $-category.
    We write $ \P^{\ex}(\cat{C}) $ for the
    $ \infty $-category $ \Fun^{\ex}(\cat{C}^{\op}, \Cat{Sp}) $,
    which is equivalent to the full subcategory of $ \P(\cat{C})
    = \Fun(\cat{C}^{\op}, \Cat{S}) $
    spanned by left exact functors.
    In this case the Yoneda embedding $ \cat{C} \to \P(\cat{C}) $
    factors through $ \P^{\ex}(\cat{C}) $. We call the functor $ \cat{C}
    \to \P^{\ex}(\cat{C}) $ the \emph{stable Yoneda embedding}.
\end{definition}

There is also an additive version
of the Yoneda embedding.
Let $ \Cat{Sp}_{\geq 0} $ denote the large $ \infty $-category of
connective spectra.
The $ \infty $-category $ \Cat{Sp}_{\geq 0} $
is a full subcategory of the stable $ \infty $-category
$ \Cat{Sp} $ closed under coproducts, hence additive.

\begin{definition}\label{7aa54feb2f}
    Let $ \cat{A} $ be an additive $ \infty $-category.
    We write $ \P^{\add}(\cat{A}) $ for the
    $ \infty $-category $ \Fun^{\add}(\cat{A}^{\op}, \Cat{Sp}_{\geq0}) $,
    which is equivalent to the full subcategory of $ \P(\cat{A}) $ spanned by
    functors which preserve finite products
    (see for example \autocite[Corollary~2.10~(iii) and Example~5.3~(ii)]{GepnerGrothNikolaus15}).
    In this case the Yoneda embedding $ \cat{A} \to \P(\cat{A}) $
    factors through $ \P^{\add}(\cat{A}) $. We call the functor $ \cat{A}
    \to \P^{\add}(\cat{A}) $ the \emph{additive Yoneda embedding}.
\end{definition}

Applying \autocite[Lemmas~5.5.4.18--19]{HTT},
we immediately see that $ \P^{\add}(\cat{A}) $
can be regarded as a strongly reflective subcategory of
$ \P(\cat{A}) = \Fun(\cat{A}^{\op}, \Cat{S}) $,
hence is presentable.
Similarly, $ \P^{\ex}(\cat{C}) $ can be regarded as a
strongly reflective subcategory of $ \P(\cat{C}) $ for
a stable $ \infty $-category~$ \cat{C} $.

\begin{remark}\label{f60e310bc2}
    In fact, these constructions are special cases of \autocite[Definition~5.3.6.5]{HTT}.
    Let $ \cls{K} $, $ \cls{K}^{\ex} $ and~$ \cls{K}^{\add} $
    denote the classes of all simplicial sets which are small,
    finite and finite discrete, respectively.
    Comparing the constructions given here with
    that in the proof of \autocite[Proposition~5.3.6.2]{HTT},
    we can see that $ \P^{\ex} $ and~$ \P^{\add} $ are
    the restrictions of $ \P^{\cls{K}}_{\cls{K}^{\ex}} $
    and~$ \P^{\cls{K}}_{\cls{K}^{\add}} $, respectively.
    As a consequence, we can view them as functors:
    \begin{align*}
        \P^{\ex} &\colon \Cat{Cat}_{\infty}^{\ex} \to \Cat{Pr}^{\textnormal{L},\ex}, &
        \P^{\add} &\colon \Cat{Cat}_{\infty}^{\add} \to \Cat{Pr}^{\textnormal{L},\add}.
    \end{align*}
    Here $ \Cat{Pr}^{\L,\ex} $ and $ \Cat{Pr}^{\L,\add} $ denote
    full subcategories of
    the very large $ \infty $-category $ \Cat{Pr}^{\L} $
    (see \autocite[Definition~5.5.3.1]{HTT} for the definition)
    spanned by exact and additive presentable $ \infty $-categories,
    respectively.
\end{remark}

\begin{lemma}\label{3265f779d1}
    Let $ \cat{A} $ be an additive $ \infty $-category.
    Then
    the exact functor
    $
        \Fun^{\add}(\cat{A}^{\op}, \Cat{Sp})
        \to \Sp(\P^{\add}(\cat{A}))
    $
    induced by
    the truncation functor
    $
        \Fun^{\add}(\cat{A}^{\op}, \Cat{Sp})
        \to \Fun^{\add}(\cat{A}^{\op}, \Cat{Sp}_{\geq0})
        = \P^{\add}(\cat{A})
    $
    is an equivalence of stable $ \infty $-categories.
\end{lemma}

\begin{proof}
    Since the full subcategory $ \P^{\add}(\cat{A}) \subset \Fun(\cat{A}^{\op}, \Cat{Sp}_{\geq0}) $ is
    closed under limits,
    under the equivalence $ \Fun(\Cat{S}^{\fin}_{*},
    \Fun(\cat{A}, \Cat{Sp}_{\geq0})) \simeq \Fun(\Cat{S}^{\fin}_{*}\times\cat{A}^{\op},
    \Cat{Sp}_{\geq0}) $
    we can regard the stable $ \infty $-category
    $ \Sp(\P^{\add}(\cat{A})) $ as the full subcategory of the $ \infty $-category
    $ \Fun(\Cat{S}^{\fin}_{*}\times\cat{A}^{\op},
    \Cat{Sp}_{\geq0}) $ consisting of functors $ f \colon
    \Cat{S}^{\fin}_{*}\times\cat{A}^{\op} \to
    \Cat{Sp}_{\geq0} $ which are reduced and excisive in the first variable and
    additive in the second variable.
    Moreover,
    since the truncation functor
    $ \tau_{\geq0} \colon \Cat{Sp} \to \Cat{Sp}_{\geq0} $
    induces a equivalence $ \Cat{Sp} \simeq \Sp(\Cat{Sp}_{\geq0}) $,
    this
    $ \infty $-category
    is equivalent to the $ \infty $-category
    $ \Fun^{\add}(\cat{A}^{\op}, \Cat{Sp}) $
    under the equivalence $ \Fun(\Cat{S}^{\fin}_{*}\times\cat{A}^{\op},
    \Cat{Sp}_{\geq0}) \simeq \Fun
    (\cat{A}^{\op}, \Fun(\Cat{S}^{\fin}_{*},
    \Cat{Sp}_{\geq0})) $.
\end{proof}

\subsection{Symmetric monoidal structure}

For the basic theory
of symmetric monoidal $ \infty $-categories,
we refer readers to \autocite{HA}.

\begin{definition}\label{0517e3b814}
    We call a symmetric monoidal $ \infty $-category~$ \cat{A}^{\otimes} $
    \emph{additive symmetric monoidal} if the underlying
    $ \infty $-category~$ \cat{A} $ is additive and tensor product operations are
    additive in each variable.

    We call a symmetric monoidal $ \infty $-category~$ \cat{C}^{\otimes} $
    \emph{stable symmetric monoidal} if the underlying $ \infty $-category~$ \cat{C} $
    is stable and tensor product operations are
    exact in each variable.
\end{definition}

We present additive and stable symmetric monoidal versions of
the ($ \infty $-categorical) Yoneda embedding,
which might be of independent interest.
We note that these ``additive and stable monoidal Yoneda embeddings''
are different from what Nikolaus called
by the same name in \autocite[Section~6]{Nikolaus16}.

\begin{proposition}\label{c9c36f3726}
    Let $ \cat{A}^{\otimes} $ and~$ \cat{C}^{\otimes} $ be
    additive and stable symmetric monoidal $ \infty $-categories
    respectively.
    Then the following hold:
    \begin{partenum}
        \item\label{201baaeda2}
            The additive $ \infty $-category $ \P^{\add}(\cat{A}) $ admits
            an additive symmetric monoidal structure
            whose tensor product operations preserve colimits in each variable.
            Moreover, there exists a symmetric monoidal functor
            $ \cat{A}^{\otimes} \to \P^{\add}(\cat{A})^{\otimes} $ whose
            underlying functor is the additive Yoneda embedding.
        \item\label{0663712d40}
            The stable $ \infty $-category $ \P^{\ex}(\cat{C}) $ admits
            a stable symmetric monoidal structure
            whose tensor product operations preserve colimits in each variable.
            Moreover, there exists a symmetric monoidal functor
            $ \cat{C}^{\otimes} \to \P^{\ex}(\cat{C})^{\otimes} $ whose
            underlying functor is the stable Yoneda embedding.
    \end{partenum}
\end{proposition}

\begin{proof}
    By \cref{f60e310bc2} and
    \autocite[Remark~4.8.1.9]{HA},
    this is a corollary of \autocite[Proposition~4.8.1.10]{HA}.
\end{proof}

Combining this proof with \autocite[Proposition~4.8.1.5]{HA}
and \autocite[Corollary~5.5~(ii)]{GepnerGrothNikolaus15},
we get the following counterpart for the construction of \cref{3265f779d1}:

\begin{corollary}\label{c3edc72e8f}
    Let $ \cat{A}^{\otimes} $ be an
    additive symmetric monoidal $ \infty $-category.
    Then the stable $ \infty $-category $ \Sp(\P^{\add}(\cat{A})) $
    admits a stable symmetric monoidal structure
    whose tensor product operations preserve colimits in each variable and
    there exists a symmetric monoidal refinement
    of the composition
    $ \cat{A} \to \P^{\add}(\cat{A}) \to \Sp(\P^{\add}(\cat{A})) $.
\end{corollary}

By construction, these three can be seen as functors
between appropriate $ \infty $-categories,
but we only need their $ 1 $-categorical functoriality in this paper.

\begin{remark}\label{7e82e3b1bc}
    Let $ \cat{A}^{\otimes} $ be
    an additive symmetric monoidal $ \infty $-category.
    By mimicking the proof of \autocite[Proposition~4.9]{Nikolaus16},
    we obtain a symmetric monoidal structure
    on $ \P^{\add}(\mathcal{A}) $ concretely
    as a localization of the symmetric monoidal structure on $ \P(\mathcal{A}) $
    of \autocite[Section~3]{Glasman16}.
    Since the symmetric monoidal structure on $ \P^{\add}(\mathcal{A}) $ of \cref{c9c36f3726}
    can be characterized by the property that
    the additive Yoneda embedding is symmetric monoidal
    and the tensor product operations preserve colimits in each variable,
    these two constructions coincide.

    Similarly,
    the symmetric monoidal structure on $ \Sp(\P^{\add}(\cat{A})) $ can be obtained by
    localizing that on $ \Fun(\mathcal{A}^{\op}, \mathsf{Sp}) $.
    In particular,
    for a stable symmetric monoidal $ \infty $-category~$ \cat{C} $,
    the functor $ \Sp(\P^{\add}(\cat{C})) \simeq \Fun^{\add}(\cat{C}^{\op},\Cat{Sp})
    \to \Fun^{\ex}(\cat{C}^{\op},\Cat{Sp}) = \P^{\ex}(\cat{C}) $
    has a symmetric monoidal refinement.
\end{remark}

\section{Weight structures}\label{cca8147ee1}

\subsection{Basic definitions and properties}

We present the basic theory of weight structures here.
We refer readers to \autocite{Bondarko10} for a detailed study
of weight structures.

\begin{definition}\label{8d18b2095c}
    Let $ \cat{D} $ be a triangulated category.
    A \emph{weight structure} on~$ \cat{D} $ is a pair of
    full subcategories $ (\cat{D}_{\w \geq 0}, \cat{D}_{\w \leq 0}) $
    satisfying the following conditions:
    \begin{enumerate}
        \item
            $ \cat{D}_{\w\geq0} $ and~$ \cat{D}_{\w\leq0} $ are
            closed under retracts in~$ \cat{D} $. In particular,
            they are closed under isomorphism.
        \item
            We have inclusions
            $ \cat{D}_{\w\geq0}[1] \subset \cat{D}_{\w\geq0} $ and
            $ \cat{D}_{\w\leq0}[-1] \subset \cat{D}_{\w\leq0} $.
        \item
            For $ X \in \cat{D}_{\w\leq0} $ and
            $ Y \in \cat{D}_{\w\geq0}[1] $, we have
            $ \Hom_{\cat{D}}(X,Y) = 0 $.
        \item
            For any $ Z \in \cat{D} $, there exists a distinguished triangle
            $ X \to Z \to Y $ where $ X \in \cat{D}_{\w\leq0} $ and
            $ Y \in \cat{D}_{\w\geq0}[1] $.
    \end{enumerate}
    If $ \cat{D} $ is equipped with a weight structure, we will write
    $ \cat{D}_{\w\geq n} $ and~$ \cat{D}_{\w\leq n} $ for
    $ \cat{D}_{\w\geq 0}[n] $ and~$ \cat{D}_{\w\leq 0}[n] $, respectively.
\end{definition}

\begin{definition}\label{33bfeb429c}
    Let $ \cat{C} $ be a stable $ \infty $-category.
    A \emph{weight structure} on~$ \cat{C} $ is a weight structure on the
    homotopy category~$ \h\cat{C} $.
    When $ \cat{C} $ is equipped with a weight structure, we will write
    $ \cat{C}_{\w\geq n} $ and~$ \cat{C}_{\w\leq n} $ for
    the full subcategories of~$ \cat{C} $ determined by
    $ \h\cat{C}_{\w\geq n} $ and~$ \h\cat{C}_{\w\leq n} $, respectively.
\end{definition}

\begin{definition}\label{6a013133ac}
    Let $ \cat{C} $ be a stable $ \infty $-category
    equipped with a weight structure.
    \begin{partenum}
        \item
            The \emph{heart}~$ \cat{C}_{\w}^{\heartsuit} $ of the weight structure
            is the full subcategory
            $ \cat{C}_{\w\geq0} \cap \cat{C}_{\w\leq0} \subset \cat{C} $.
        \item
            We denote the full subcategory of~$ \cat{C} $ consisting of
            objects $ X $ satisfying
            $ X \in \cat{C}_{\w\geq m} \cap \cat{C}_{\w\leq n} $ for
            some $ m $, $ n $
            by~$ \cat{C}^{\b} $.
            The weight structure on~$ \cat{C} $ is called \emph{bounded}
            if the equality $ \cat{C}^{\b} = \cat{C} $ holds.
    \end{partenum}
\end{definition}

\begin{remark}\label{c83efbfb70}
    Unlike the case of $ \t $-structures,
    the heart of a weight structure is generally not equivalent to
    an ordinary category, i.e., the mapping spaces may not be
    (homotopy) discrete.
\end{remark}

\begin{example}\label{6a1432a5bc}
    Let $ \cat{B} $ be an additive category.
    Then the stable $ \infty $-category
    $ \Kb(\cat{B}) = \Ndg(\operatorname{Ch}^{\b}(\cat{B})) $
    has a canonical weight structure,
    where $ \Ndg $ denotes the dg nerve construction
    given in \autocite[Construction~1.3.1.6]{HA}.
    In this case, the weight structure is bounded by definition.
    Its heart is the essential image under
    the canonical embedding $ \cat{B} \to \Kb(\cat{B}) $.
\end{example}

We present some basic facts concerning weight structures here.

\begin{lemma}\label{cec9699fb4}
    Let $ \cat{C} $ be a stable $ \infty $-category
    equipped with a weight structure.
    Then the following hold:
    \begin{partenum}
        \item\label{f63d91cbb8}
            The heart~$ \cat{C}_{\w}^{\heartsuit} $ is
            an additive subcategory of~$ \cat{C} $.
        \item\label{04c07dfbe3}
            The full subcategory~$ \cat{C}^{\b} $
            is closed under finite limits and colimits.
            In particular, $ \cat{C}^{\b} $ is a stable $ \infty $-category.
            Moreover, the pair
            $ (\cat{C}^{\b} \cap \cat{C}_{\w \geq 0}, \cat{C}^{\b} \cap \cat{C}_{\w \leq 0}) $
            defines a weight structure on~$ \cat{C}^{\b} $,
            whose heart coincides with that of~$ \cat{C} $.
        \item\label{8ab91f1f25}
            If the weight structure is bounded,
            the $ \infty $-category~$ \cat{C} $
            is generated by the heart~$ \cat{C}_{\w}^{\heartsuit} $
            under finite limits and colimits.
    \end{partenum}
\end{lemma}

\begin{proof}
    \Cref{f63d91cbb8} is immediate from the definition.
    \Cref{04c07dfbe3,8ab91f1f25}
    are $ \infty $-categorical reformulations of
    \autocite[Proposition~1.3.6]{Bondarko10} and
    \autocite[Corollary~1.5.7]{Bondarko10}, respectively.
\end{proof}

\subsection{Weight complex}

Before stating Sosnilo's result,
which we use to
define the weight complex functor,
we give a definition of a morphism between stable $ \infty $-categories
equipped with weight structures.

\begin{definition}\label{02a8228d1d}
    Let $ \cat{C} $ and~$ \cat{C}' $ be stable $ \infty $-categories
    equipped with weight structures.
    A functor $ f \colon \cat{C} \to \cat{C}' $ is called
    \emph{weight exact} if it is exact and carries
    $ \cat{C}_{\w\geq 0} $ and~$ \cat{C}_{\w\leq 0} $ into
    $ \cat{C}'_{\w\geq 0} $ and~$ \cat{C}'_{\w\leq 0} $, respectively.
\end{definition}

The following result is due to Sosnilo:

\begin{proposition}[Sosnilo]\label{53341890c5}
    Let $ \cat{C} $ be a stable $ \infty $-category equipped with
    a bounded weight structure.
    Then the following hold:
    \begin{partenum}
        \item\label{96ce752920}
            The composition
            \begin{equation*}
                \cat{C}
                \to \P^{\ex}(\cat{C})
                = \Fun^{\ex}(\cat{C}^{\op}, \Cat{Sp})
                \to \Fun^{\add}((\cat{C}_{\w}^{\heartsuit})^{\op}, \Cat{Sp})
            \end{equation*}
            is exact and fully faithful.
            Here the second arrow is given by restriction.
        \item\label{559c32c4d0}
            Let $ \cat{C}' $ be a stable $ \infty $-category
            equipped with a weight structure. Then
            the restriction functor
            $
                \Fun^{\wex}(\cat{C}, \cat{C}')
                \to \Fun^{\add}(\cat{C}_{\w}^{\heartsuit},
                \cat{C}'^{\heartsuit}_{\w})
            $
            is an equivalence of $ \infty $-categories.
    \end{partenum}
\end{proposition}

\begin{proof}
    See \autocite[Proposition~3.3]{Sosnilo17}.
\end{proof}

Let $ \cat{C} $ be a stable $ \infty $-category equipped with
a bounded weight structure.
Combining
\cref{96ce752920} of \cref{53341890c5}
and \cref{8ab91f1f25} of \cref{cec9699fb4},
we can identify $ \cat{C} $ with the full subcategory
$ \Sp(\P^{\add}(\cat{C}_{\w}^{\heartsuit})) $ generated by
the essential image of the embedding
$ \cat{C}_{\w}^{\heartsuit} \to \Sp(\P^{\add}(\cat{C}_{\w}^{\heartsuit})) $
under finite limits and colimits.
We denote this subcategory by $ \Sp(\P^{\add}(\cat{C}_{\w}^{\heartsuit}))^{\fin} $;
beware that this $ \infty $-category is not determined by
$ \Sp(\P^{\add}(\cat{C}_{\w}^{\heartsuit})) $ itself
but by the additive $ \infty $-category~$ \cat{C}_{\w}^{\heartsuit} $.
Under the hypothesis of~\cref{559c32c4d0} of \cref{53341890c5},
we can describe
the weight exact functor $ g \colon \cat{C} \to \cat{C}' $ which corresponds
to an additive functor $ f \colon
\cat{C}_{\w}^{\heartsuit} \to
\cat{C}'^{\heartsuit}_{\w} $ by the equivalence
under this identification:
The functor~$ g $ can be regarded as
the restriction of the functor
$ \Sp(\P^{\add}(\cat{C}_{\w}^{\heartsuit})) \to \Sp(\P^{\add}(\cat{C'}^{\heartsuit}_{\w})) $
determined by $ f $
to the full subcategory
$ \Sp(\P^{\add}(\cat{C}_{\w}^{\heartsuit}))^{\fin} $.

Now we define the weight complex functor as follows:

\begin{definition}\label{e28ce92cea}
    Let $ \cat{C} $ be a stable $ \infty $-category equipped with a
    bounded weight structure. The \emph{weight complex} functor is
    the weight exact functor $ \cat{C} \to \Kb(\h \cat{C}_{\w}^{\heartsuit}) $
    which is mapped to (an additive functor equivalent to)
    the additive functor $ \cat{C}_{\w}^{\heartsuit}
    \to \h \cat{C}_{\w}^{\heartsuit} $
    under the equivalence $ \Fun^{\wex}(\cat{C}, \Kb(\h\cat{C}_{\w}^{\heartsuit}))
    \to \Fun^{\add}(\cat{C}_{\w}^{\heartsuit},
    \h \cat{C}^{\heartsuit}_{\w}) $ of \cref{559c32c4d0} of \cref{53341890c5}.
\end{definition}

\begin{remark}\label{ad3f437300}
    As shown in \autocite[Corollary~3.5]{Sosnilo17},
    the functor of \cref{e28ce92cea} is an $ \infty $-categorical
    enhancement of the functor
    constructed in \autocite{Bondarko10} under the name
    ``strong weight complex functor''.
\end{remark}

\section{Main theorem}\label{17df43aa2c}

To state our main result, we need a notion of compatibility
between a stable symmetric monoidal structure and a weight structure.

\begin{definition}\label{948985e80f}
    Let $ \cat{C}^{\otimes} $ be
    a stable symmetric monoidal $ \infty $-category.
    We call a weight structure $ (\cat{C}_{\w\geq0}, \cat{C}_{\w\leq0}) $ on
    the underlying $ \infty $-category~$ \cat{C} $ \emph{compatible} with
    the symmetric monoidal structure if $ \cat{C}_{\w\geq0} $ and~$ \cat{C}_{\w\leq0} $
    are closed under tensor product operations.
\end{definition}

By \autocite[Proposition~2.2.1.1]{HA},
if the underlying $ \infty $-category of a stable symmetric monoidal
$ \infty $-category~$ \cat{C}^{\otimes} $ has a compatible weight structure,
the symmetric monoidal structure on~$ \cat{C}^{\otimes} $
can be restricted to
its full subcategories $ \cat{C}_{\w\geq0} $,
$ \cat{C}_{\w\leq0} $, and~$ \cat{C}_{\w}^{\heartsuit} $.
We let
$ (\cat{C}_{\w}^{\heartsuit})^{\otimes} $
denote the restriction to the heart.

\begin{lemma}\label{1c256e6a74}
    Let $ \cat{C}^{\otimes} $ be a stable symmetric monoidal
    $ \infty $-category whose underlying $ \infty $-category is equipped with
    a bounded compatible weight structure.
    Then the fully faithful functor
    $
        \cat{C}
        \to \Sp(\P^{\add}(\cat{C}_{\w}^{\heartsuit}))
    $
    given in~\cref{96ce752920} of \cref{53341890c5} admits
    a canonical symmetric monoidal refinement.
\end{lemma}

\begin{proof}
    We consider the following diagram
    of symmetric monoidal $ \infty $-categories commutative up to homotopy:
    \begin{equation*}
        \xymatrix{
            (\cat{C}_{\w}^{\heartsuit})^{\otimes}
            \ar[r]^{i^{\otimes}}
            \ar[d]_{j^{\otimes}}
            & \cat{C}^{\otimes}
            \ar[d]_{j'^{\otimes}}
            \ar[dr]^{j''^{\otimes}}
            \\
            \Sp(\P^{\add}(\cat{C}_{\w}^{\heartsuit}))^{\otimes}
            \ar[r]^{I^{\otimes}}
            & \Sp(\P^{\add}(\cat{C}))^{\otimes}
            \ar[r]^(.55){L^{\otimes}}
            & \P^{\ex}(\cat{C})^{\otimes}
            \rlap{.}
        }
    \end{equation*}
    The square is constructed
    from the inclusion
    $ i^{\otimes} \colon (\cat{C}_{\w}^{\heartsuit})^{\otimes} \to \cat{C}^{\otimes} $
    using the functoriality of the construction $ \Sp(\P^{\add}(\text{--}))^{\otimes} $.
    The functor~$ j''^{\otimes} $ is the stable symmetric monoidal
    Yoneda embedding for~$ \cat{C}^{\otimes} $ and
    the functor~$ L^{\otimes} $ is a symmetric monoidal refinement of the reflector~$ L $
    (see \cref{7e82e3b1bc}).
    Note that the functors~$ I $ and~$ L $ have right adjoints
    whose composition is the restriction functor
    $ \P^{\ex}(\cat{C}) = \Fun^{\ex}(\cat{C}^{\op}, \Cat{Sp})
    \to \Fun^{\add}((\cat{C}_{\w}^{\heartsuit})^{\op}, \Cat{Sp})
    \simeq \Sp(\P^{\add}(\cat{C}_{\w}^{\heartsuit})) $.
    In particular, both functors are exact.

    Let $ \cat{C}' $ be the full subcategory
    of $ \P^{\ex}(\cat{C}) $
    generated by the
    essential image of
    the functor $ j'' \circ i \simeq L \circ I \circ j $ under finite limits and colimits.
    Since the functors~$ j'' $, $ L $ and~$ I $ are exact,
    $ \cat{C}' $ is equal to the essential image of~$ j'' $
    and also that of the restriction of~$ L \circ I  $ to the full subcategory
    $ \Sp(\P^{\add}(\cat{C}_{\w}^{\heartsuit}))^{\fin} $.
    From the former equality
    and the symmetric monoidality of $ j''^{\otimes} $,
    we can see that $ \cat{C}' $ is closed under
    the tensor product operations in $ \Sp(\P^{\add}(\cat{C})) $.
    Hence we can give a symmetric monoidal structure on~$ \cat{C}' $
    by restriction.

    Since $ L^{\otimes} $ and~$ I^{\otimes} $ are symmetric monoidal,
    the restriction of $ L^{\otimes} \circ I^{\otimes} $ defines
    a symmetric monoidal functor
    $ f^{\otimes} \colon
    (\Sp(\P^{\add}(\cat{C}_{\w}^{\heartsuit}))^{\fin})^{\otimes} \to \cat{C}'^{\otimes} $,
    where the left hand side
    denotes the restriction of $ \Sp(\P^{\add}(\cat{C}_{\w}^{\heartsuit}))^{\otimes} $ to
    the full subcategory.
    According to~\cref{96ce752920} of \cref{53341890c5},
    the right adjoint of $ L \circ I $ induces an equivalence
    $ \cat{C}' \to \Sp(\P^{\add}(\cat{C}_{\w}^{\heartsuit}))^{\fin} $,
    so we deduce that
    its underlying functor $ f $ is an equivalence,
    which means that $ f^{\otimes} $ is itself an equivalence
    by \autocite[Remark~2.1.3.8]{HA}.
    Therefore, the composition
    \begin{equation*}
        \cat{C}^{\otimes}
        \xrightarrow[\sim]{j''^{\otimes}}
        \cat{C'}^{\otimes}
        \xleftarrow[\sim]{f^{\otimes}}
        (\Sp(\P^{\add}(\cat{C}_{\w}^{\heartsuit}))^{\fin})^{\otimes}
        \subset
        \Sp(\P^{\add}(\cat{C}_{\w}^{\heartsuit}))^{\otimes}
    \end{equation*}
    is the desired functor.
\end{proof}

Given a stable $ \infty $-category~$ \cat{C} $ equipped with a bounded weight structure
and a symmetric monoidal structure on its heart~$ \cat{C}_{\w}^{\heartsuit} $,
we can construct a stable symmetric monoidal structure on~$ \cat{C} $
by restricting the symmetric monoidal structure on $ \Sp(\P^{\add}(\cat{C}_{\w}^{\heartsuit})) $
to its full subcategory $ \Sp(\P^{\add}(\cat{C}_{\w}^{\heartsuit}))^{\fin}
\simeq \cat{C} $,
which is closed under the tensor product operations in $ \Sp(\P^{\add}(\cat{C}_{\w}^{\heartsuit})) $.
In particular, we can attach a symmetric monoidal structure to
the stable $ \infty $-category $ \Kb(\cat{B}) $
of \cref{6a1432a5bc}
for an additive symmetric monoidal category~$ \cat{B}^{\otimes} $.
\Cref{1c256e6a74} ensures that this construction is the only way to do this,
if it is required that the symmetric monoidal structure on~$ \cat{C} $ is
compatible with its weight structure and its restriction to
the heart~$ \cat{C}_{\w}^{\heartsuit} $
coincides with the given symmetric monoidal structure.

\begin{theorem}\label{b39219127c}
    Let $ \cat{C}^{\otimes} $ and~$ \cat{C}'^{\otimes} $ be a stable
    symmetric monoidal $ \infty $-categories whose underlying $ \infty $-categories
    are equipped with bounded compatible weight structures.
    Let $ f \colon \cat{C} \to \cat{C}' $ be a weight exact functor
    whose restriction to their hearts
    $ \cat{C}_{\w}^{\heartsuit} \to \cat{C}'^{\heartsuit}_{\w} $ admits
    a symmetric monoidal refinement
    $ g^{\otimes} \colon (\cat{C}_{\w}^{\heartsuit})^{\otimes}
    \to (\cat{C}'^{\heartsuit}_{\w})^{\otimes} $.
    Then $ f $ admits
    a canonical symmetric monoidal refinement
    whose restriction to their hearts coincides with~$ g^{\otimes} $.
\end{theorem}

\begin{proof}
    Let
    $ G^{\otimes} \colon \Sp(\P^{\add}(\cat{C}_{\w}^{\heartsuit}))^{\otimes}
    \to \Sp(\P^{\add}(\cat{C}'^{\heartsuit}_{\w}))^{\otimes} $
    be the symmetric monoidal functor determined by~$ g^{\otimes} $
    and the functoriality of the construction $ \Sp(\P^{\add}(\text{--}))^{\otimes} $.
    According to \cref{1c256e6a74} and the description of the correspondence of~\cref{559c32c4d0}
    given in the paragraph after the proof of \cref{53341890c5},
    the restriction
    $ (\Sp(\P^{\add}(\cat{C}_{\w}^{\heartsuit}))^{\fin})^{\otimes}
    \to (\Sp(\P^{\add}(\cat{C}'^{\heartsuit}_{\w}))^{\fin})^{\otimes} $
    of~$ G^{\otimes} $ can be
    regarded as a symmetric monoidal refinement of the functor~$ f $.
\end{proof}

\begin{remark}\label{eb827e5652}
    By taking the full subcategory of $ \cat{C}'^{\otimes} $
    spanned by objects which can be written as
    $ Y_1 \otimes \dotsb \otimes Y_n $
    for some $ Y_1 $, \dots, $ Y_n \in \cat{C}'^{\b} $ up to equivalence,
    we can see that the conclusion of \cref{b39219127c} still holds when
    the weight structure on $ \cat{C}' $ is not necessarily bounded.
\end{remark}

Applying this theorem to the case
$ \cat{C}'^{\otimes} = \Kb(\h\cat{C}_{\w}^{\heartsuit})^{\otimes} $,
we have the following main result of this paper:

\begin{corollary}\label{7704dfff4b}
    For a stable symmetric monoidal $ \infty $-category whose underlying
    stable $ \infty $-category has a bounded weight structure compatible
    with the symmetric monoidal structure,
    the weight complex functor has a symmetric monoidal refinement.
\end{corollary}

\printbibliography

\end{document}